\documentclass[10 pt]{amsart}

\usepackage{latexsym, amsmath, amssymb, longtable, booktabs,amscd,microtype,booktabs,cases}
\usepackage[square]{natbib}
\usepackage{graphicx}
\usepackage[dvipsnames]{xcolor}
\usepackage{caption}
\usepackage{tcolorbox}
\usepackage{tikz}
\usepackage{dsfont}
\usepackage[all]{xy}
\usepackage{enumerate}
\usepackage{multirow}

\usepackage[colorinlistoftodos,prependcaption, textsize=tiny]{todonotes}
\reversemarginpar

\usepackage{hyperref}
\hypersetup{
    colorlinks,
    citecolor=blue,
    filecolor=black,
    linkcolor=red,
    urlcolor=black
}
\bibpunct{[}{]}{,}{n}{}{;} 

\numberwithin{equation}{section}
\usepackage{mathtools}

\DeclarePairedDelimiterX{\infdivx}[2]{}{}{%
	#1\;\delimsize\|\;#2%
}
\newcommand{\infdiv}{\infdivx}

\newcommand{\Z}{\mathbb{Z}}

\newcommand{\Ga}{\Gamma}

\newcommand{\C}{\mathbb{C}}

\newcommand\N{\mathbb{N}}

\newcommand\sO{\mathcal{O}}

\newcommand\GL{{\mathrm{GL}}}

\newcommand{\Q}{\mathbb{Q}}

\newcommand\Ind{{\mathrm{Ind}}}

\newtheorem{thm}{Theorem}[section]
\newtheorem{theorem}[thm]{Theorem}
\newtheorem{cor}[thm]{Corollary}

\newtheorem{prop}[thm]{Proposition}
\newtheorem{lemma}[thm]{Lemma}
\theoremstyle{definition}
\newtheorem{definition}[thm]{Definition}
\newtheorem{remark}[thm]{Remark}

\theoremstyle{definition}

\theoremstyle{remark}

\theoremstyle{remark}

\makeatletter
\def\imod#1{\allowbreak\mkern10mu({\operator@font mod}\,\,#1)}
\makeatother

	\title{A note on quadratic twisting of epsilon factors for modular forms with arbitrary nebentypus}
	\author{Debargha Banerjee}
	\author{Tathagata Mandal}
	\address{INDIAN INSTITUTE OF SCIENCE EDUCATION AND RESEARCH, PUNE, INDIA}
\begin{document}

\begin{abstract}
    In this article, we investigate the variance of local $\varepsilon$-factor for a  modular
    form with arbitrary nebentypus with respect to  twisting by a quadratic character. We detect the type of the supercuspidal representation 
    from that. 
    For modular forms with trivial nebentypus, similar results are proved by Pacetti
    \cite{MR3056552}. Our method however is completely different from that of Pacetti 
    and we use representation theory crucially. 
    For ramified principal series (with $\infdiv{p}{N}$ and $p$ odd) and 
    unramified supercuspidal representations of level zero, we relate 
    these numbers with the Morita's $p$-adic Gamma function. 
    \end{abstract}
	
	\maketitle

\section{Introduction}

Pacetti~\cite{MR3056552} studied the variance 
 of local root numbers in the context of twisting by a quadratic character 
 for modular forms with {\it trivial nebentypus} and
 determined the type of local automorphic representations at $p$ as an application. 
 In this article,  we explore the same
 and investigate what properties of modular forms with {\it arbitrary nebentypus}
 are encoded therein. 
 In particular, we determine the type of the local component 
 $\pi_{f,p}$ (for each prime $p \mid N$) of the automorphic representation
 $\pi_f$ attached to $f$ from that [cf. Corollary~\ref{mainkoro}].
 We also give a criteria for a modular form to be $p$-minimal
 [cf. Definition~\ref{pminimal}] in terms of the parity of $N_p$ (the exact 
 power of $p$ that divides $N$)
 [cf. Proposition~\ref{np} and Proposition~\ref{n2}]
 
 The ramification formulae 
 of the endomorphism algebras of motives attached to non-CM Hecke eigenforms 
 for all supercuspidal primes are given in \cite{DT}. 
 The statement of the main theorem in the above mentioned article 
 depends on the nature of the supercuspidal primes  [cf. \S \ref{supdef}] . 
 We endeavor to determine the same that appear as  local components of elliptic Hecke eigenforms by analyzing the variance of the local factors by twisting.
 In another direction, Pacetti and his co-authors found applications of the computation 
 in the context of Heegner/ Darmon points \cite{MR3729491}. Our results will be useful
 in a similar context for modular forms with arbitrary nebentypus.  
Following~\cite{MR2869056}, we will be mostly interested in the case when $\pi_{f,p}$ is supercuspidal.

 The local $\varepsilon$-factors depend on additive characters chosen 
 [cf. Section~\ref{preli}].
 Pacetti used an additive character of conductor zero.
 In this present paper, we choose an additive character  
 of conductor $-1$ for non-supercuspidal representations and
 of conductor zero or any for supercuspidal representations. 
 Note that the global $\varepsilon$-factor which is a product 
 of all local $\varepsilon$-factors does not depend on 
 additive characters \cite[Section $3.5$]{MR546607}.

 We classify the quadratic extensions $K$ of $\Q_p$  
 in the dihedral [cf. \S \ref{supdef}] supercuspidal primes in terms of the variation 
 of global $\varepsilon$-factor with respect to twisting by a quadratic character
 [cf. Corollary~\ref{mainkoro} and Corollary~\ref{corop=2}].
 Our method is completely different from that of Pacetti 
 as we relied on a theorem due to Deligne [cf. Theorem~\ref{dlgn}].
 The above mentioned theorem is not applicable for unramified dihedral supercuspidal prime
 $p$ with $a(\chi)=1$ [cf. Section~\ref{statement}]. In this situation
 and principal series representation with $\infdiv{p}{N}$ and $p$ odd, we relate 
 the variance of the local $\varepsilon$-factor with Morita's $p$-adic
 Gamma function [cf. \S \ref{padicgamma}].

 Using the local inverse theorem \cite[Section $27$]{MR2234120}, it is possible 
 to determine the case where $\pi_{f,p}$ is supercuspidal by the variation
 of the local $\varepsilon$-factors  by twisting with respect to a certain set of characters
 and it is less convenient from computational perspective~\cite{MR2869056}.
 We manage to do the same by a suitable {\it quadratic} twist.   We give a complete classification if the corresponding local Weil-Deligne representations are non-dihedral for $p=2$.

\subsection{Acknowledgements}
The first named author was partially supported by the SERB grant YSS/2015/ \\
001491 and the second named author was supported by the IISER Pune Ph.D fellowship.

\section{Preliminaries} \label{preli}
 \subsection{Notation}
 For a non-archimedian local field $F$ of characteristic zero,
 let $\mathcal{O}_F$ denote the ring of integers in $F$,
 $\mathfrak{p}_F$ the maximal ideal in $\mathcal{O}_F$ and 
 $k_F=\mathcal{O}_F / \mathfrak{p}_F$ the residue field of $F$.
 The $m$-th principal units of $F$ is denoted by $U_F^m=1+\mathfrak{p}_F^m$.
 Let $v_F$ be a valuation on $F$.
 For the local field $\Q_p$, we will denote them by $\Z_p,\mathfrak{p}_p,k_p,U_p^m$
 and $v_p$ respectively. The norm and trace maps from $F$ to $\Q_p$ are denoted 
 respectively by $N_{F|\Q_p}$ and $\mathrm{Tr}_{F|\Q_p}$.
 We denote the set of multiplicative (respectively  additive) characters of $F$  
 by $\widehat{F^\times}$ (respectively $\widehat{F}$).
 
 For any quadratic extension $K|F$ and $x \in F^\times$,
 the symbol $(x, K|F)= 1$ or $-1$ according as $x$ is a norm  
 of an element of $K$ or not.
 
 \begin{definition}
 	The level $l(\chi)$ of a non-trivial quasi-character $\chi$ of $F^\times$
 	is the smallest positive integer $m \geq 0$ such that
 	$\chi(U_F^{m+1})=1$. We say the conductor of $\chi$ is $m+1$
 	and it is denoted by $a(\chi)$.
 	 \end{definition}
 A character $\chi$ is called \emph{unramified} if 
 	the conductor is zero,
 	\emph{tamely ramified} if it has conductor $1$ and
 	\emph{wildly ramified} if its conductor is greater or equal to $2$.
 	For $\chi_1,\chi_2 \in \widehat{F^\times}$, we have
 	$a(\chi_1\chi_2) \leq \text{max}(a(\chi_1),a(\chi_2))$.
 	The equality holds if $a(\chi_1) \neq a(\chi_2)$.
 	For a non-trivial additive character $\phi$ of $F$,  the conductor
 	$n(\phi)$ is the smallest integer such that $\phi$ is
 	trivial on $\mathfrak{p}_F^{-n(\phi)}$.

 Let $\mathbb{F}_q$ denote a finite field of order $q=p^r$. The classical Gauss sum
 $G(\chi,\phi)$ associated to a multiplicative character 
 $\chi$ of $\mathbb{F}_q^\times$ and 
 an additive character $\phi$ of $\mathbb{F}_q$ is defined by
 \[
 G(\chi,\phi)=\sum_{x \in \mathbb{F}_q^\times}^{} \chi(x) \phi(x).
 \]
 We will denote it by $G_r(\chi)$ as the additive character $\phi$ is fixed.
 For $\chi \in \widehat{\mathbb{F}_q^\times}$ and 
 $\phi \in \widehat{\mathbb{F}_q}$, let $\chi'=\chi \circ N_{\mathbb{F}_q|\mathbb{F}_p}$
 and $\phi'=\phi \circ \mathrm{Tr}_{\mathbb{F}_q|\mathbb{F}_p}$ denote their lifts
 to $\mathbb{F}_q$. Then by the Davenport-Hasse theorem
 \cite[Theorem~$11.5.2$]{MR1625181}, we have
 $G(\chi',\phi')=(-1)^{r-1}G(\chi,\phi)^r$. In our notation, we simply write it as
 $G_r(\chi')=(-1)^{r-1}G_1(\chi)^r$. 
 \subsection{Dihedral and non-dihedral supercuspidal representations}
 \label{supdef}
As mentioned in the introduction, the main technical novelty of this paper are in the case when 
$\pi_{f,p}$ is supercuspidal. 
 Let $  \rho_p(f): W(\Q_p) \to \GL_2(\C)$
 be the local representation of the Weil-Deligne group $W(\Q_p)$ 
 associated to $\pi_{f,p}$.
\begin{definition}
\label{defdihedral}
 In the supercuspidal case, we call a prime $p$ to be \emph{dihedral}
 for the modular form $f$ if the representation is induced 
 by a character $\chi$ of an index two subgroup $W(K)$ of $W(\Q_p)$. 
 \end{definition}
 Depending on $K$ unramified (or ramified), we say $p$ is an unramified
 (or ramified) supercuspidal prime for $f$.

If $p=2$, there are supercuspidal representations that are not induced by a character; we call it {\it non-dihedral} supercuspidal representations. 
 
\subsection{$p$-adic Gamma function} \label{padicgamma}
 For an odd prime $p$ and $z \in \Z_p$, define the $p$-adic gamma function
 \cite[Chapter~$7$]{MR1760253} to be
 \[
   \Ga_p(z)=\lim_{n \to z} (-1)^n \prod_{0<j<n, (p,j)=1} j,
 \]
 where $n$ tends to $z$ $p$-adically through positive integers. 
 Let $\chi$ be a multiplicative character of $\mathbb{F}_p$ of order $k$.
 Using the Gross-Koblitz formula,
 we deduce that \cite[Corollary~$3.1$]{MR1014384}:
 \begin{eqnarray} \label{GKcoro}
    G_1(\chi^r)=(-p)^{r/k} 
    \Ga_p \big( \frac{r}{k} \big).
 \end{eqnarray}

 For a given non-trivial additive character $\phi$ of $F$
 and a Haar measure $dx$ on $F$, the $L$-function corresponding to  
 a quasi-character $\chi$ of $F^\times$ satisfies a functional equation.
 This defines a number $\varepsilon(\chi,\phi,dx) \in \C^\times$ 
 \cite[Section $3$]{MR546607}.

 \subsection{Local $\varepsilon$-factors}
 The local $\varepsilon$-factor associated to a non-trivial character $\chi$ of $F^\times$
 and a non-trivial character $\phi$ of $F$ is defined as follows \cite[p. $5$]{RPL}:
 \[
   \varepsilon(\chi, \phi, c) = \chi(c)
   \frac{\int_{U_F}^{} \chi^{-1}(x) \phi(\frac{x}{c}) dx}
   {|\int_{U_F}^{} \chi^{-1}(x) \phi(\frac{x}{c}) dx|},
 \]
 where $c \in F^\times$ has valuation $a(\chi)+n(\phi)$.
 Here, we consider the normalized Haar measure $dx$ on $F$.
 The above formula can be simplified as \cite[p. $94$]{MR0457408}:
 \begin{eqnarray} \label{equ1}
   \varepsilon(\chi, \phi, c) =q^{- \frac{a(\chi)}{2}}
   \sum_{x \in \frac{U_F}{U_F^{a(\chi)}}}^{} \chi^{-1}(\frac{x}{c}) 
   \phi(\frac{x}{c})
   =q^{- \frac{a(\chi)}{2}} \chi(c) \tau(\chi, \phi),
 \end{eqnarray}
 where
$\tau(\chi, \phi)=  \sum_{x \in \frac{U_F}{U_F^{a(\chi)}}}^{} \chi^{-1}(x) \phi(\frac{x}{c})$.
 This is called the local Gauss sum associated to the characters $\chi$
 and $\phi$. It is independent of the coset representatives $x$ chosen.
 The element $c$ in the formula~(\ref{equ1}) is determined by its valuation
 up to a unit $u$. It can be shown that $\varepsilon(\chi, \phi, c)=
 \varepsilon(\chi, \phi, cu)$. Thus, for simplicity we write
 $\varepsilon(\chi, \phi, c)=\varepsilon(\chi, \phi)$.

 If $\chi$ is unramified, then the valuation $v_F(c)=n(\phi)$
 and thus we have $\varepsilon(\chi,\phi,c)=\chi(c)$.
 When $a(\chi)=1$, the local Gauss sum turns out to be the classical Gauss sum.
 Since $\chi$ is tamely ramified,
 $\widetilde{\chi}:=\chi^{-1}|_{\sO_F^\times}$ can be considered as a 
 character of $\sO_F^\times/U_F^1 \cong k_F^\times$.
 If we take an additive character $\phi$ of $F$ with $n(\phi)=-1$,
 then we can choose $c=1$. In this settings, the local Gauss sum coincides
 with the well-known classical Gauss sum.
 
 We now list some basic properties of local $\varepsilon$-factors which can be found
 in \cite{MR546607}.
 \begin{enumerate}
 	\item 
 	     $\varepsilon(\chi, \phi_a)=\chi(a)|a|_F^{-1} \varepsilon(\chi,\phi)$,
 	     where $a \in F^\times$ and $\phi_a(x)=\phi(ax)$. 
 	     Here, $|\,\,|_F$ denote the absolute value of $F$.
 	\item
 	     $\varepsilon(\chi\theta, \phi)=\theta(\pi_F)^{a(\chi)+n(\phi)}
 	     \varepsilon(\chi,\phi)$, where $\theta$ is an unramified character of $F^\times$.
 	     The element $\pi_F$ is a uniformizer of $F$.
    \item
         $\varepsilon \big(\Ind_{W(F)}^{W(\Q_p)} \rho, \phi \big)
         =\varepsilon \big(\rho, \phi \circ \mathrm{Tr}_{F|\Q_p} \big)$,
         where $\rho$ denote a virtual zero dimensional representation 
         of a finite extension $F|\Q_p$.
 \end{enumerate}
 The local $\varepsilon$-factor $\varepsilon(\chi,\phi)$ depends on 
 the additive character $\phi$ chosen which follows from the property $(1)$ above.

 Let $\chi$ denote the quadratic character attached to the quadratic
 extension of $\Q$ ramified only at $p$. 
 For odd primes $p$, $\Q$ has two quadratic extensions ramified only at $p$,
 namely $\Q(\sqrt{p})$ and $\Q(\sqrt{-p})$. 
 For $p=2$, there are three quadratic extensions ramified only at $2$, namely $\Q(i),\Q(\sqrt{2})$ and $\Q(\sqrt{-2})$. By class field theory, 
 the character $\chi$ can be identified with a character of the id\`{e}le group, that is, 
 characters $\{\chi_q\}_q$ with $\chi_q: \Q_q^\times \to \C^\times$
 satisfying the following conditions:
 \begin{enumerate}
 	\item
 	     For distinct primes $q \neq p$, the character $\chi_q$ is 
 	     unramified and $\chi_q(q)=\Big( \frac{q}{p}\Big)$.
 	\item 
 	    $\chi_p$ is ramified with conductor $p$ and the restriction
     	$\chi|_{\Z_p^\times}$ factors through the unique quadratic
 	    character of $\mathbb{F}_p^\times$ with $\chi_p(p)=1$.
 \end{enumerate}
 By definition, we say that $\chi_p$ is tamely ramified.
 In this article, we study the changes of the local factors
 associated to $f$ while twisting by $\chi$. Let $\varepsilon_p$
 denote the variation of the local factor of $f$ at $p$
 while twisting by $\chi_p$. On both sides, we choose the same 
 additive character and Haar measure.

 \begin{lemma} \label{chip}
 	Let $\chi_p$ be the quadratic character of $\Q_p^\times$ as above. 
 	Then for an additive character $\phi$ of $\Q_p$ of conductor $-1$, we
 	have 
 	\begin{eqnarray*} 
 		\varepsilon(\chi_p,\phi)=
 		\begin{cases}
 			1, & \quad \text{if} \,\, p \equiv 1 \pmod{4}, \\
 			i, & \quad \text{if} \,\, p \equiv 3 \pmod{4}.
 		\end{cases}
 	\end{eqnarray*}
 	For $p=2$, we have $\varepsilon(\chi_2,\phi)=2^{-1/2}$.
 \end{lemma}

 \begin{proof}
 	Since $\chi_p$ is tamely ramified, $\widetilde{\chi}_p:=\chi_p^{-1}|_{\Z_p^\times}$
 	becomes a character of $\mathbb{F}_p^\times$.
 	Let $\phi_p, \widetilde{\phi}_p$ denote the canonical additive character of $\Q_p,
 	\mathbb{F}_p$ respectively. Note that $\phi$ can be written as $a \cdot \phi_p$,
 	for some $a \in \Q_p^\times$ \cite[\S $1.7$ Proposition, p. $11$]{MR2234120}. 
 	We will find out a proper element $a$ such that
 	$\phi|_{\Z_p}=a \cdot \phi_p|_{\Z_p}$ induces the canonical additive character 
 	of $\mathbb{F}_p$. \cite[Lemma $3.1$]{S} ensures us that $1/p$ can be taken 
 	as a value of $a$. By the property $(1)$ of local $\varepsilon$-factors, 
 	we have
 	\[
 	\varepsilon(\chi_p,\phi)=\varepsilon(\chi_p, \frac{1}{p}\phi_p)=
 	\chi_p(\frac{1}{p}) \varepsilon(\chi_p,\phi_p)
 	\overset{(\ref{equ1})}{=}
 	p^{-1/2} \tau(\chi_p,\phi_p).
 	\]
 	Now $\tau(\chi_p,\phi_p)$ turns out to be the classical Gauss sum 
 	$G(\widetilde{\chi}_p, \widetilde{\phi}_p)$.
 	Using \cite[Theorem $5.15$]{MR1294139}, we have 
 	\begin{eqnarray*} 
 		G(\widetilde{\chi}_p, \widetilde{\phi}_p) = 
 		\begin{cases}
 			p^{1/2}, & \quad \text{if} \,\, p \equiv 1 \pmod{4}, \\
 			ip^{1/2}, & \quad \text{if} \,\, p \equiv 3 \pmod{4},
 		\end{cases}
 	\end{eqnarray*}
 	where $i$ is a fourth primitive root of unity.
 	Combining all above, we obtain the required result for odd primes.
 	For $p=2$, the Gauss sum $G(\widetilde{\chi}_2, \widetilde{\phi}_2)=1$. 
 	Therefore, we get that $\varepsilon(\chi_2,\phi)=2^{-1/2}$.
 \end{proof}

 For an additive character $\phi$ of $\Q_p$, the induced character on $F$ is denoted 
 by $\phi_F=\phi \circ 
 \mathrm{Tr}_{F|\Q_p}$. For all $c \in F$, consider the additive character $\phi_{F,c}(x)=\phi_F(cx)$.
 \begin{lemma} \label{dlemma} 
 	Let $\chi \in \widehat{F^\times}$ and $\phi_F \in \widehat{F}$ be two 
 	non-trivial characters.
 	Let $r \in \N$ be such that $2r \geq a(\chi)$. 
 	Then there is an element $c \in F^\times$ with valuation $-(a(\chi)+n(\phi_F))$ 
 	such that
 	\begin{equation} \label{chir}
 	   \chi(1+x)=\phi_F(cx) \,\, \forall \,\, x \in \mathfrak{p}_F^r.
 	\end{equation}
 \end{lemma}
 
 \begin{proof}
 	Since $2r \geq a(\chi)$, the character $\chi$ satisfy the relation:
 	$\chi(1+x)\chi(1+y)=\chi(1+x+y)$, for all $x,y \in \mathfrak{p}_F^r$.
 	This is same to having that the map $x \mapsto \chi(1+x)$ is an additive 
 	character on $\mathfrak{p}_F^r$ which can be extended to $F$. 
 	
 	By the property of
 	local additive duality \cite[\S $1.7$  Proposition, p. $11$]{MR2234120},
 	the set $\{\phi_{F,c}: c\in F \}$ is the group of all characters of $F$.
 	Hence, there exists an element $c \in F^\times$ such that
 	\[
 	\chi(1+x)=\phi_F(cx) \,\, \forall \,\, x \in \mathfrak{p}_F^r.
 	\]
 	Using the same proposition,
 	the conductor of the character $\phi_{F,c}$ is $-(n(\phi_F)+v_F(c))$.
 	From the equality of the conductors of both sides, we get the desired 
 	valuation of $c$.
 \end{proof}
 We now recall a fundamental result of Deligne about the behavior of local factors 
 while twisting.
 \begin{theorem} \cite[Lemma$~4.1.6$]{MR0349635}   \label{dlgn}
 	Let $\alpha$ and $\beta$ be two quasi-characters of $F^\times$ 
 	such that $a(\alpha) \geq 2a(\beta)$.
 	If $\alpha(1+x)=\phi_F(cx)$ for $x \in \mathfrak{p}_F^r$ with $2r \geq 
 	a(\alpha)$ (if $a(\alpha)=0$, then $c=\pi_F^{-n(\phi_F)}$), then
 	\[
 	\varepsilon(\alpha \beta, \phi_F)=\beta^{-1}(c) \varepsilon(\alpha,\phi_F).
 	\]
 \end{theorem}
 \noindent
 Note that the valuation of $c$ is $-(a(\alpha)+n(\phi_F))$ in the above theorem.
 
 We write the level $N$ of the newform $f$ as $p^{N_p}N'$, with $p \nmid N'$ and 
 the nebentypus $\epsilon = \epsilon_p \cdot \epsilon'$
 as a product of characters of $(\Z/p^{N_p}\Z)^\times$ and $(\Z/N'\Z)^\times$
 of conductors $p^{C_p}$ for some $C_p \leq N_p$ and $C'$ dividing $N'$
 respectively.
 
 \begin{definition} \cite[p. $236$]{MR508986}
 	\label{pminimal}
 	We say $f$ is $p$-minimal, if the $p$-part of its level is the smallest among 
 	all twists $f \otimes \chi$ of $f$ by Dirichlet characters $\chi$.
 \end{definition}

\section{Statement of results} \label{statement}
 We denote the group $\mathrm{GL}_2(\Q_p)$ by $G$.
 Let $\mu_1,\mu_2$ be two quasi-characters of $\Q_p^\times$ and  $V(\mu_1,\mu_2)$
 be the space of locally constant functions $\psi: G \to \C$
 with the following property:
 \[
 \psi \Big(
 \left[ {\begin{array}{cc}
 	a & * \\
 	0 & d \\
 	\end{array} } \right]
 g
 \Big)
 =
 \mu_1(a) \mu_2(d) |a/d|^{1/2} \psi(g),
 \]
 for all $a,d \in \Q_p^\times$ and $g \in G$.
 The induced representation of $G$ by its action on $V(\mu_1,\mu_2)$
 through right translation is denoted by $\rho(\mu_1,\mu_2)$.

 One knows that $\rho(\mu_1,\mu_2)$ is irreducible except when
 $\mu_1 \mu_2^{-1}=|\,\,|^{\pm 1}$. In this case, the representation 
 $\rho(\mu_1,\mu_2)$ is called a \emph{principal series representation},
 denoted by $\pi:=\pi(\mu_1,\mu_2)$.

 The induced representation $\rho(\mu_1,\mu_2)$ is not irreducible if and only if
 $\mu_1 \mu_2^{-1}=|\,\,|^{\pm 1}$. The unique irreducible sub-representation 
 of $\rho(|\,\,|^{1/2}, |\,\,|^{-1/2})$ is the \emph{Steinberg representation},
 denoted by St. More generally, we may assume that
 \begin{eqnarray} \label{chi12}
 \mu_1=\mu |\,\,|^{1/2}, \quad \mu_2=\mu |\,\,|^{-1/2},
 \end{eqnarray}
 for some character $\mu$ of $\Q_p^\times$. In this case $\rho(\mu |\,\,|^{1/2},
 \mu |\,\,|^{-1/2})$ contains a unique irreducible sub-representation which is the
 twist $\mu \cdot \mathrm{St}$ of the Steinberg representation. This representation
 $\mu \cdot \mathrm{St}$ is called a \emph{special representation}, again denoted 
 by $\pi:=\pi(\mu_1,\mu_2)$. The resulting factor is the one dimensional 
 representation $\mu \circ \mathrm{det}$ of $G$.

 The local $\varepsilon$-factor
 of a special representation $\pi$ is given by $\varepsilon_T(\pi,s,\phi)=
 \varepsilon_T(\mu_1,s,\phi) \varepsilon_T(\mu_2,s,\phi) E(\mu_1,\mu_2,s)$
 \cite[Table, p. $113$]{MR0379375} with
 \begin{eqnarray} \label{E}
 E(\mu_1,\mu_2,s) = 
 \begin{cases}
 1, & \quad \text{if} \,\, \mu_1 \,\, \text{is ramified}, \\
 -\mu_2(p)p^{-s}, & \quad \text{otherwise}.
 \end{cases}
 \end{eqnarray}
 For a ramified character $\mu$ of $\Q_p$ of level $n \geq 0$, $s \in \C$ and an additive character
 $\phi$ of $F$ with $n(\phi)=-1$, the local rational function
 $\varepsilon_T(\mu,s,\phi) \in \C(p^{-s})$ is defined by
 \cite[Equ. $23.6.2$]{MR2234120}
 \[
   \varepsilon_T(\mu,s,\phi) =p^{n(\frac{1}{2}-s)}
   \mu(c) \tau(\mu,\phi)/p^{(n+1)/2}.
 \]
 This function is called 
 the \emph{Tate local constant} of $\mu$ associated to $\phi$.
 The local epsilon factors and Tate local constants are related by the following relation:
 $\varepsilon_T(\mu, \frac{1}{2},\phi)=\varepsilon(\mu,\phi)$.

 Let $\omega_p$ be 
 the $p$-part of the central character of $\pi_f$,
 $a_p(f)$ be the $p$-th Fourier coefficient of $f$ and 
 $\mu_1$ be an  unramified character with
 $\mu_1(p)=a_p(f)/p^{(k-1)/2}$. We also consider a character 
 $\mu_2$ with $\mu_1 \mu_2=\omega_p$. 
 In the ramified principal series case, $\omega_p$ has conductor 
 $p^{N_p}$ with $N_p \geq 1$.
 When $f$ is a $p$-minimal form with $\pi_{f,p}$ is a ramified principal series representation,
 we have $\pi_{f,p} \cong \pi(\mu_1,\mu_2)$
 \cite[prop. $2.8$]{MR2869056}. 
 In this case, choose an additive character $\phi$ of conductor $-1$ satisfying 
 \begin{eqnarray}
  \label{omegap}
    \omega_p(1+x)=\phi(cx) \,\, \forall \,\, x \in \mathfrak{p}_p^r,
 \end{eqnarray}
 with $2r \geq a(\omega_p)$. 
 
 If  $N_p=1$, let $m$  be the order of 
 $\widetilde \omega_p:=\omega_p^{-1}|_{\Z_p}^\times$. 
 Consider the quantity that depends on $p$ and $m$:
 \begin{eqnarray*} 
 	c_p= \begin{cases}
 		(-p)^{-1/2m} \{\Ga_p(\frac{1}{2m}) / \Ga_p(\frac{1}{m}) \}, & \quad
 	    \text{if} \,\, p \,\, \text{odd with} \,\,
 	    N_p=1 \,\, \text{and} \,\, m \,\, \text{even}, \\
 		1, & \quad	\text{otherwise}.
 	\end{cases}
 \end{eqnarray*}
 
 With the choice of an additive character as in equation~\ref{omegap}, 
 we prove the following theorem:
 \begin{theorem} \label{psr}
 	Let $\pi_{f,p}$ be a ramified principal series representation. 
 	Choose an additive character $\phi$ as above. 
 	For odd primes $p$, the number
 	\begin{eqnarray*} 
 		\varepsilon_p = 
 		\begin{cases}
 		    p^{\frac{1-k}{2}}a_p(f) c_p, & \quad \text{if} \,\, p \equiv 1 \pmod{4}, \\
 			ip^{\frac{1-k}{2}}a_p(f)c_p, & \quad \text{if} \,\, p \equiv 3 \pmod{4}.
 		\end{cases}
 	\end{eqnarray*} 
	For $p=2$, we have $\varepsilon_2=2^{- \frac{k}{2}} a_2(f)$.
 \end{theorem}

 We now consider the case where $\pi_{f,p}$ is a special representation.
 If $f$ is a $p$-primitive newform, then by \cite[prop. $2.8$]{MR2869056},
 $\pi_{f,p} \cong \mu \cdot \mathrm{St}$, where $\mu$ is unramified
 and $\mu(p)=a_p(f)/p^{(k-2)/2}$. 
 Hence, for $j=1,2$, the character $\mu_j$ in (\ref{chi12}) is unramified.
 With an additive character $\phi$ of conductor $-1$,
 our result in this case is as follows:
 \begin{theorem} \label{sp}
 	If $\pi_{f,p}$ is a special representation, then the number $\varepsilon_p$ is 
 	\begin{eqnarray*} 
 		\varepsilon_p = 
 		\begin{cases}
 			-p^{\frac{3-k}{2}}a_p(f), & \quad \text{if} \,\, p \equiv 1 \pmod{4}, \\
 			p^{\frac{3-k}{2}}a_p(f), & \quad \text{if} \,\, p \equiv 3 \pmod{4}.
 		\end{cases}
 	\end{eqnarray*}
 	For $p=2$, we have $\varepsilon_2=-2^{\frac{1-k}{2}} a_2(f)$.
 \end{theorem}

 Every irreducible admissible representation of $G$ that is not a sub-representation 
 of some $\rho(\mu_1,\mu_2)$ is called a \emph{supercuspidal representation}.
 Note that supercuspidal cases are the most interesting cases in the computation
 of the local data of a modular form \cite[Section $2$]{MR2869056}.
 The following is the main result proved in this case:
 \begin{theorem} \label{main}
 	Let $p$ be a dihedral supercuspidal prime for $f$.
 	\begin{enumerate}
 		\item
 		Let $K|\Q_p$ be unramified. 
 		For an additive character $\phi$ of $\Q_p$ as in (\ref{chir}) with $a(\chi) >1$,
 		we have $\varepsilon_p=1$. 
 		\item
 		Assume that $p$ is odd with $K|\Q_p$ is ramified. We have
 		\begin{itemize}
 			\item
 			$\varepsilon_p=1$ if the conductor of $\chi$ is odd.
 			\item
 			In the case of $a(\chi)$ even, the number
 			\begin{equation} 
 			\varepsilon_p =
 			\begin{cases}
 			1, 
 			\quad \text{if} \,\, (p,K|\Q_p)=1,  \\
 			\Big( \frac{-1}{p} \Big),
 			\quad \text{if} \,\, (p,K|\Q_p)=-1.
 			\end{cases}
 			\end{equation}
 		\end{itemize}
 		\item
 		     When $p=2$ with $K|\Q_2$ ramified, we have
 		     \begin{equation} \label{eps2}
 		        \varepsilon_2 =
 		        \begin{cases}
 		           1, \quad \text{if} \,\, K=\Q_2(\sqrt{-1}), \Q_2(\sqrt{2}),
 		           \Q_2(\sqrt{-2}),  \\
 		           -1, \quad \text{if} \,\, K=\Q_2(\sqrt{3}) 
 		           \,\, \text{with} \,\, \chi \,\, \text{not minimal}
 		           \,\, \text{or if} \,\,
 		           K=\Q_2(\sqrt{6}), \Q_2(\sqrt{-6}) \,\,
 		           \text{with} \,\, \chi \,\, \text{minimal},  \\
 		           1, \quad \text{if} \,\, K=\Q_2(\sqrt{3}) 
 		           \,\, \text{with} \,\, \chi \,\, \text{minimal}
 		           \,\, \text{or if} \,\,
 		           K=\Q_2(\sqrt{6}), \Q_2(\sqrt{-6}) \,\,
 		           \text{with} \,\, \chi \,\, \text{not minimal}.
 		     \end{cases}
 		\end{equation}
 	\end{enumerate}
 \end{theorem}
 
 The condition $a(\chi)$ odd or even is determined by Proposition~\ref{np}.
 For the definition of minimality of $\chi$, see $\S$ \ref{sup}.
 In the unramified case, the above theorem is not valid for $a(\chi)=1$. In that case,
 $\widetilde{\chi}:=\chi^{-1}|_{\sO_K^\times}$ can be considered as a 
 character of $\sO_K^\times/U_K^1 \cong \mathbb{F}_{p^2}^\times$
 and the associated local Gauss sum turns out to be  
 the classical Gauss sum. Here, we take an additive character
 $\phi$ of $K$ which induces the canonical additive character $\widetilde{\phi}$
 of $\mathbb{F}_{p^2}$.

 \begin{theorem} \label{evenodd}
 	Let $p$ be an odd unramified supercuspidal prime with $a(\chi)=1$.
 	\begin{enumerate}
 		\item 
 		If the order of $\widetilde \chi$ is even, then $\varepsilon_p=1$.
 		\item
 		Assume the order $m$ of $\widetilde \chi$ is odd that divides $p-1$.
 		Write $p=bm+1$ for some $b \in \N$. Then
 		\begin{eqnarray} \label{gap}
 		\varepsilon_p=
 		p^{-1/m} \Big\{ \Ga_p \big(\frac{1}{2m} \big) / 
 		\Ga_p \big(\frac{1}{m} \big) \Big\}^2.
 		\end{eqnarray}
 	\end{enumerate}
 \end{theorem}
 
 \begin{theorem} \label{appstkl}
 	If $p$ is an odd unramified supercuspidal prime with $a(\chi)=1$
 	and the order $m$ of $\widetilde{\chi}$ divides $p+1$, then
 	\begin{eqnarray} 
 	   \varepsilon_p=
 	   \begin{cases}
 	      -1, & \quad \text{if} \,\, m \,\,
 	      \text{odd and} \,\, p \equiv 1 \pmod{4}, \\
 	      1, & \quad \text{if} \,\, m \,\,
 	      \text{even and} \,\, p \equiv 1 \pmod{4},   \\
 	      1, & \quad \text{if} \,\, m \,\,
 	      \text{even and} \,\, p \equiv 3 \pmod{4} \,\, \text{with} \,\, 
 	      \frac{p+1}{m} \,\, \text{odd}.
 	   \end{cases}
 	\end{eqnarray}
 	For $p=2$, we have $\varepsilon_2=1$.
 \end{theorem}
  
\section{Non-supercuspidal representations}
 \noindent
 \textbf{Proof of Theorem~\ref{psr}.}
 By \cite[Table, p. $113$]{MR0379375}, the local factor 
 associated to $f$ is $\varepsilon(\mu_1,\phi) \varepsilon(\mu_2, \phi)$
 and the local $\varepsilon$-factor corresponding to $f \otimes \chi_p$ is 
 $\varepsilon(\mu_1\chi_p,\phi) \varepsilon(\mu_2\chi_p, \phi)$.
 Since $\mu_1$ is unramified, the local $\varepsilon$-factors are computed as follows:
 \begin{enumerate}
 	\item 
 	$\varepsilon(\mu_1,\phi)=\mu_1(c)=\mu_1(\frac{1}{p})=\frac{1}{\mu_1(p)}$.
 	\item
 	$\varepsilon(\mu_2,\phi)=\varepsilon(\mu_1^{-1}\omega_p,\phi)=
 	\mu_1^{-1}(p)^{a(\omega_p)-1} \varepsilon(\omega_p,\phi)$,
 	by property $(2)$ of local $\varepsilon$-factors.
 	\item
 	$\varepsilon(\mu_1\chi_p,\phi)=
 	\mu_1(p)^{a(\chi_p)-1} \varepsilon(\chi_p,\phi)=
 	\varepsilon(\chi_p,\phi)$ [Again, by property $(2)$ of local $\varepsilon$-factors and since the conductor of $\chi_p$ is $1$].
 	\item
	For $N_p >1$, note that $a(\omega_p) \geq 2 a(\chi_p)$.  In this case, we compute that
 	\begin{eqnarray*}
 	   \varepsilon(\mu_2\chi_p, \phi)
 	   &=&
 	   \varepsilon(\mu_1^{-1}\omega_p\chi_p, \phi)  \\
 	   &=&
 	   \mu_1^{-1}(p)^{a(\omega_p)-1} \varepsilon(\omega_p \chi_p, \phi)  \quad
 	   (\text{by property} \,\, (2) \,\, \text{of local} \,\, \varepsilon-\text{factors})
 	   \\
 	   &\overset{\text{Theorem}~(\ref{dlgn})}{=}&
 	   \mu_1^{-1}(p)^{a(\omega_p)-1} \chi_p^{-1}(c) 
 	   \varepsilon(\omega_p,\phi)  \\
 	   &=&
 	   \mu_1^{-1}(p)^{a(\omega_p)-1} \varepsilon(\omega_p,\phi) \quad
 	   (\text{since} \,\, \chi_p(c)=1).
 	\end{eqnarray*}
 	Thus, we deduce that $\varepsilon(\mu_2,\phi)=\varepsilon(\mu_2\chi_p,\phi)$.
 	
 	Now assume that $N_p=1$. Both $\widetilde \omega_p:=\omega_p^{-1}|_{\Z_p}^\times$
 	and $\widetilde \chi_p$ can be thought of as a character of $\mathbb{F}_p^\times$.
 	Notice that $\widehat{\mathbb{F}_p^\times}$ is cyclic, say $\widehat{\mathbb{F}_p^\times}= \langle \chi_1 \rangle$.

	If $\widetilde \omega_p$ has even order, then both $\widetilde \omega_p$
 	and $\widetilde \omega_p \widetilde \chi_p$ have same order.
 	Hence, we can write $\widetilde \omega_p=\widetilde \omega_p \widetilde \chi_p
 	=\chi_1^a$ for some $a$. As a result, we obtain 
 	$\varepsilon(\omega_p,\phi)=\varepsilon(\omega_p\chi_p,\phi)$.

	If $\widetilde \omega_p$ has odd order $m$, then $\widetilde \omega_p
 	\widetilde \chi_p$ has order $2m$. 
 	Write $p=bm+1$ for some $b \in \N$.
 	Thus, we have $\widetilde \omega_p=\chi_1^b$
 	and $\widetilde \omega_p \widetilde \chi_p=\chi_1^{\frac{b}{2}}$.
 	By the formula~(\ref{GKcoro}), we obtain
 	$G_1(\chi_1^b)=(-p)^{\frac{b}{p-1}} \Ga_p \big(\frac{b}{p-1}\big)
 	=(-p)^{\frac{1}{m}} \Ga_p \big(\frac{1}{m} \big)$
 	and $G_1(\chi_1^{\frac{b}{2}})
 	=(-p)^{\frac{1}{2m}} \Ga_p \big(\frac{1}{2m} \big)$.
 	Therefore, we deduce that 
 	$\varepsilon(\omega_p \chi_p, \phi)= \varepsilon(\omega_p,\phi) \cdot 
 	(-p)^{-1/2m} \{\Ga_p(\frac{1}{2m}) / \Ga_p(\frac{1}{m})\}$.
 	For $p=2$, we have 
 	$\varepsilon(\omega_p \chi_p, \phi)= \varepsilon(\omega_p,\phi)$.
\end{enumerate}
 From above, we compute that
 \[
   \varepsilon_p=\frac{\varepsilon(\mu_1\chi_p,\phi) \varepsilon(\mu_2\chi_p, \phi)}
   {\varepsilon(\mu_1,\phi) \varepsilon(\mu_2, \phi)} =c_p \cdot 
   \frac{\varepsilon(\mu_1\chi_p,\phi)}{\varepsilon(\mu_1,\phi)}=
   c_p \cdot \mu_1(p) \varepsilon(\chi_p,\phi).
 \]
 Using Lemma~\ref{chip}, we get the required result.

\noindent
\textbf{Proof of Theorem~\ref{sp}.}
 Note that both $\mu_1,\mu_2$ are unramified characters. Thus, we have that
 $\varepsilon(\mu_j,\phi)=\mu_j(c)=\mu_j(\frac{1}{p})=1/\mu_j(p)$, for all $j=1,2$.
 Also, for $j=1,2$, we have $\varepsilon(\mu_j \chi_p,\phi)=\mu_j(p)^{a(\chi_p)-1} 
 \varepsilon(\chi_p,\phi)=\varepsilon(\chi_p,\phi)$ as $a(\chi_p)=1$.
 By Lemma~\ref{chip}, we have
 \begin{eqnarray*} 
 	\varepsilon(\mu_j \chi_p,\phi) = 
 	\begin{cases}
 		1, & \quad \text{if} \,\, p \equiv 1 \pmod{4}, \\
 		i, & \quad \text{if} \,\, p \equiv 3 \pmod{4},
 	\end{cases}
 \end{eqnarray*}
 for all $j=1,2$. Since the number $\varepsilon_p$ is
 \[
  \varepsilon_p=\frac{\varepsilon(\mu_1 \chi_p,\phi) \varepsilon(\mu_2 \chi_p,\phi)
  	            E(\mu_1\chi_p,\mu_2\chi_p, \frac{1}{2})}
                {\varepsilon(\mu_1,\phi) \varepsilon(\mu_2,\phi)
                E(\mu_1,\mu_2, \frac{1}{2})},
 \]
 using the relation~(\ref{E}), we get the result.
 \section{Supercuspidal representations}
  \label{sup}
 
 By local Langlands correspondence the representations $\pi_{f,p}$ are in a bijection
 with (isomorphism classes of) complex two dimensional Frobenius-semisimple
 Weil-Deligne representations $\rho_p(f)$ associated to a modular form $f$
 at $p$. For more details of Weil-Deligne representations, we refer to
 \cite[Section $3$]{MR0347738}. We will be using the information about $\rho_p(f)$
 in this case.
 
\subsection{The case $p$ odd}
 Let $\rho_p(f)$ denote the local representation of the Weil-Deligne group
 $W(\Q_p)$ attached to $f$ at a prime $p$. In the supercuspidal case,
 \[
 \rho_p(f) = \Ind_{W(K)}^{W(\Q_p)} \chi
 \]
 with $K$ a quadratic extension of $\Q_p$ and $\chi$  a quasi-character 
 of $W(K)^{\text{ab}}$ which does not factor through the norm map
 with a quasi-character of $W(\Q_p)^\text{ab}$.
 We can consider $\chi$ as a character of $K^\times$ via the isomorphism
 $W(K)^{\text{ab}} \simeq K^\times$ and say that
 $(K,\chi)$ is an admissible pair attached to $f$ at $p$. 
 It satisfies the following conditions:
 \begin{enumerate}
 	\item 
 	$\chi$ does not factor through the norm map 
 	$N_{K|\Q_p} : K^\times \to \Q_p^\times$ and
 	\item
 	if $K|\Q_p$ is ramified, then the restriction $\chi |_{U_K^1}$
 	does not factor through $N_{K|\Q_p}$.
 \end{enumerate}
 The pair $(K,\chi)$ is said to be minimal 
 if $\chi|_{U_K^{l(\chi)}}$ does not factor through the norm map $N_{K|\Q_p}$.
 If $\chi$ is minimal over $\Q_p$, then we have $a(\chi) \leq a(\theta_K \chi)$
 for all characters $\theta$ of $\Q_p^\times$. The induced character on $K$ is denoted by $\theta_K=\theta \circ N_{K|\Q_p}$.
 Clearly, $f$ is $p$-minimal if and only if its associated admissible pair is minimal.
 For more details of an admissible pair, we refer to \cite[Section $18$]{MR2234120}.

 By the properties of local $\varepsilon$-factors, recall the formula 
 for the conductor of the supercuspidal representations~\cite[Theorem $2.3.2$]{MR1913897}:
 \begin{eqnarray} \label{relation}
 a(\mathrm{Ind}_{W(K)}^{W(\Q_p)} \chi)=v_p(\delta(K|\Q_p)) +f(K|\Q_p)a(\chi).
 \end{eqnarray}
 Here,  the normalized valuation of $\Q_p^\times$ is denoted by $v_p$ and 
 $\delta(K|\Q_p),f(K|\Q_p)$ denote the discriminant and  the residual degree of $K|\Q_p$ respectively. 
 The above formula is same as the formula for the Artin conductor 
 of a $2$-dimensional induced representation of $\mathrm{Gal}(\overline{\Q}_p|\Q_p)$
 \cite[Proposition $4(b)$, p. $158$]{CasselFrohlich}.

  \begin{definition} \cite[\S $13.4$ Definition]{MR2234120}
 	An element $\alpha \in K^\times$ with $v_K(\alpha)=-n$
 	is said to be minimal over $\Q_p$
 	if one of the following holds:
 	\begin{enumerate}
 		\item 
 		     $K|\Q_p$ is ramified and $n$ is odd;
 		\item
 		     $K|\Q_p$ is unramified and the field extension $k_K|k_p$
 		     is generated by the coset $p^n\alpha+\mathfrak{p}_K$.
 	\end{enumerate}
 \end{definition}

 \begin{lemma} \cite[\S $18.2$ Proposition]{MR2234120}
 	\label{even}
 	Let $K|\Q_p$ be a tamely ramified quadratic extension with $l(\chi)=m \geq 1$.
 	If $\alpha \in \mathfrak{p}_K^{-m}$ be such that
 	$ 
 	   \chi(1+x)=\phi_K(\alpha x) \,\, \text{for} \,\, x \in \mathfrak{p}_K^m,
 	$
 	then $(K,\chi)$ is a minimal admissible pair if and only if  
 	$\alpha$ is minimal over $\Q_p$.
 \end{lemma}
  Applying the above lemma when $K|\Q_p$ is ramified quadratic, we see that
 if $(K,\chi)$ is a minimal admissible pair, then $m$ is odd
 and so $a(\chi)=l(\chi)+1$ is even~\cite[Theorem $3.3$]{MR2869056}.
 
 \begin{lemma} \label{tunnel} \cite[Lemma $1.8$]{MR721997}
 	Let $E|K$ be a quadratic separable extension with
 	residue degree $f$. 
 	If $\eta$ is a quasi-character of $K^\times$,
 	then $f a(\eta \circ N_{E|K})=a(\eta)+a(\eta \omega_{E|K})-a(\omega_{E|K})$.
 	If $\psi$ is a non-trivial additive character of $K$, then
 	$n(\psi \circ \mathrm{Tr}_{E|K})=(2/f)n(\psi)+d(E|K)$.
 	Here, $\omega_{E|K}$ denote the non-trivial character of $K^\times$
 	with kernel equal to the group of norms from $E^\times$ to
 	$K^\times$.
 \end{lemma}
We can find the valuation of the level of the modular form with arbitrary nebentypus from 
 the following proposition (see also  \cite[Corollary $3.1$]{MR3056552} for $\Ga_0(N)$). The local factor in the ramified case can be computed from that. 
 
 \begin{prop} \label{np}
 	The pair $(K, \chi)$ can be characterized in terms of $N_p$.  
 	\begin{enumerate}
 		\item 
 		     If $K|\Q_p$ is unramified, then $N_p$ is even and $\chi$ is ramified.
 		\item
 		     Assume that $K|\Q_p$ is ramified. 
 		      		   We have $N_p$ is odd if $(K,\chi)$ is minimal; otherwise it is even.
 		     \end{enumerate}
 \end{prop}
 
 \begin{proof}
 	Using the relation (\ref{relation}), we get that
 	\begin{eqnarray} 
 	N_p = 
 	\begin{cases}
 	2 a(\chi), & \quad \text{if} \,\, K|\Q_p \,\,
 	\text{is unramified}, \\
 	1+a(\chi), & \quad \text{if} \,\, K|\Q_p \,\,
 	\text{is ramified}.
 	\end{cases}
 	\end{eqnarray}
 	When $K|\Q_p$ is unramified,  we have $N_p$ is even from above. 
	If $\chi$ is unramified, then
 	it would factor through the norm map.
	Hence, the character $\chi$ is ramified.  
	 	
 	In the ramified case, if $(K,\chi)$ is a minimal admissible pair, then
 	the result follows from paragraph after Lemma~\ref{even}.
 	
	We now prove that $N_p$ is even if $(K, \chi)$ is not minimal.
Consider a non-minimal pair $(K, \chi)$.
	We can write $\chi=\chi' \theta_K$ \cite[Section $18.2$]{MR2234120}
	for a character $\theta$ of $\Q_p^\times$   
	and a minimal admissible pair $(K, \chi')$. 
	Since $\chi'$ is minimal over $\Q_p$, we have
	$a(\chi') \leq a(\chi' \theta_K) = a(\chi)$ and 
	$l(\chi')$ is odd. As a result, we obtain that $a(\chi') \geq 2$ is even.

	If $a(\theta)=0$, then we have $a(\theta_K)=0$ by Lemma~\ref{tunnel}.
	Thus, we get that $a(\chi)=a(\chi' \theta_K)=a(\chi')$ 
	and hence $\chi|_{U_K^{l(\chi)}} = (\chi'\theta_K)|_{U_K^{l(\chi)}}
	=\chi'|_{U_K^{l(\chi')}}$ does not factor through the norm map, 
	contradicts the non-minimality of $\chi$.
	Therefore, $a(\theta) \geq 1$  so that  
	$a(\theta \omega_{K|\Q_p})=a(\theta)$ 
	\cite[Proof of Proposition $2.6$]{MR721997}.
	By Lemma~\ref{tunnel} and $a(\omega_{K|\Q_p})=1$ 
	\cite[Proposition-Definition $(1)$, p. $217$]{MR2234120}, we deduce that
	$a(\theta_K)=a(\theta)+a(\theta \omega_{K|\Q_p})-a(\omega_{K|\Q_p})
	=2 a(\theta)-1$ .

	Now, we claim that $a(\theta_K)>a(\chi')$. If not, then
	we have $a(\theta_K)<a(\chi')$ (since the equality is not possible as $a(\chi')$ is even and $a(\theta_K)$ is odd).
	Hence, both $\chi=\chi' \theta_K$ and $\chi'$ have same conductor, 
	which contradicts the fact that
	$(K,\chi)$ is not minimal.
	Thus, we have proved that
	$a(\chi)=a(\chi' \theta_K)=a(\theta_K)$ is odd.
	Hence, we deduce that $N_p$ is even. 
 \end{proof}
  The above proposition gives the criteria for the associated admissible pair 
 		     to be minimal in the ramified case.
 \begin{remark}
 	In the case where $K|\Q_p$ is ramified quadratic,
 	$a(\chi)=1$ is not possible by the part ($2$)
 	of the definition of an admissible pair. Thus, $N_p=2$
 	does not occur. 
 \end{remark}

 Note that $(\Ind_{W(K)}^{W(\Q_p)}\chi )\chi_p=\Ind_{W(K)}^{W(\Q_p)} (\chi\chi_p')$ with
 $\chi_p'=\chi_p \circ N_{K|\Q_p}$. Thus, by the property $(3)$
 of local $\varepsilon$-factors, we only need to treat
 the one dimensional cases. We now consider two cases depending on $K$ unramified or ramified.

\subsubsection{The case $K|\Q_p$ unramified}
 Since $\chi|_{\Z_p}=\epsilon_p^{-1}$ \cite[Equ. $4$]{MR3391026}
 and $\epsilon_p$ is a trivial character
 when $C_p=0$,  we get the following corollary \cite[part $(3)$ of Theorem $3.2$]{MR3056552}:
 \begin{cor}
 	Assume that $C_p=0$. If $p$ is an unramified supercuspidal prime, then
 	$ \varepsilon_p=-\Big( \frac{-1}{p} \Big)$.
 \end{cor}
 
 Since $K|\Q_p$ is unramified quadratic, we can take $\pi=p$ 
 as a uniformizer of $K$ which we fix now. 
 We now prove Theorem~\ref{main} in the unramified case.

 \begin{proof}
 	We apply Theorem~\ref{dlgn} with $\alpha=\chi$ and $\beta=\chi_p'$ and get the number
 	$\varepsilon_p=\chi_p'(p)^{-1}=1$.
 \end{proof}
 
 Note that the above theorem does not work when $a(\chi)=1$.
 We give a different proof of the theorem above when $a=a(\chi)=2$
 in the unramified case. 
 
 \begin{proof}
 	Choose an additive character $\phi$ of $\Q_p$
 	of conductor zero. Using Lemma \ref{tunnel}, we get $n(\phi_K)=0$.
 	Assume that $\chi(1+x)=\phi_K(cx)$, for all $x \in \mathfrak{p}_K$.
 	Every element $x \in \sO_K^\times/1+\mathfrak{p}_K^2$
 	has the form $b_0+b_1p$
 	with $b_0 \neq 0$ and $b_i \in \mathbb{F}_{p^2} \,\, \forall  \,\, i$.
 	Now
 	\begin{eqnarray*}
 	    \tau(\chi, \phi_K)
 		&=&
 		\sum_{x \in \sO_K^\times/1+\mathfrak{p}_K^2}^{}
 		\chi^{-1}(x) \phi_K(\frac{x}{p^2})  \\
 		&=&
 		\sum_{b_i}^{} \chi^{-1} \big(b_0+b_1 p \big)
 		\phi_K \big( \frac{b_0}{p^2} +
 		\frac{b_1}{p} \big)  \\
 		&=&
 		\sum_{b_0 \in \mathbb{F}_{p^2}^\times}^{} \chi^{-1}(b_0)
 		\sum_{b_1 \in \mathbb{F}_{p^2}}^{}    
 		\chi^{-1} \big( 1+\frac{b_1}{b_0} p \big)  
 	    \phi_K \big( \frac{b_0}{p^2} +\frac{b_1}{p} \big).
 	\end{eqnarray*}
 	With the choice of the additive character, we obtain that
 	\begin{eqnarray*}
 	    \tau(\chi, \phi_K) 
 		&=& 
 		\sum_{b_0 \in \mathbb{F}_{p^2}^\times}^{} \chi^{-1}(b_0)
 		\sum_{b_1 \in \mathbb{F}_{p^2}}^{}    
 		\phi_K \big(  -\frac{b_1}{b_0} \frac{1}{p} \big)  
 		\phi_K \Big( \frac{b_0}{p^2} + \frac{b_1}{p} \Big)  \\
 		&=&
 		\sum_{b_0 \in \mathbb{F}_{p^2}^\times}^{} \chi^{-1}(b_0) 
 		\phi_K \big(\frac{b_0}{p^2} \big)
 		\sum_{b_1 \in \mathbb{F}_{p^2}}^{}  
 		\phi_K \Big( \big(1-\frac{1}{b_0} \big) 
 		\frac{b_1}{p} \Big)
 	\end{eqnarray*}
 	Since the sum of a non-trivial character over a group vanishes, the inner
 	sum is zero unless $b_0=1$. As a result,
 	we obtain that $\tau(\chi,\phi_K)=p^2 \phi_K \big(\frac{1}{p^2} \big)$.
 	Since $K|\Q_p$ is unramified, $a(\chi_p')=1$ by Lemma~\ref{tunnel}.
 	Thus, in a similar way we get that 
 	$\tau(\chi \chi_p',\phi_K)=p^2 \phi_K \big(\frac{1}{p^2} \big)$, 
 	completing the proof.
 \end{proof}
 
 \noindent
 \textbf{The case $a(\chi)=1$.}
 Let $\chi$ be a tamely ramified character. We now prove Theorem~\ref{evenodd}.

 \begin{proof}
 	\begin{enumerate}
 		\item 
 		     We know that, as a group $\widehat{\mathbb{F}_{p^2}^\times} \simeq \mathbb{F}_{p^2}^\times$.
 		     Thus, $\widehat{\mathbb{F}_{p^2}^\times}$ is cyclic, say 
 		     $\widehat{\mathbb{F}_{p^2}^\times}= \langle \chi_2 \rangle$.
 		     Since $\widetilde\chi$ has even order, both $\widetilde\chi$
 		     and $\widetilde\chi \widetilde \chi_p'$ have same order.
 		     Hence, we can write $\widetilde\chi=\widetilde\chi 
 		     \widetilde \chi_p'=\chi_2^a$ for some
 		     $a \in \{1, \cdots,p^2-1\}$.  
 		     Thus, we get $\varepsilon_p=1$.
 	    \item
 	         Since $o(\widetilde{\chi}) \mid (p-1)$,
 	         the character $\widetilde{\chi}$ can be thought of
 	         as a lift of some character $\widetilde{\chi}^*$ on $\mathbb{F}_p$ and
 	         $\widetilde{\chi},\widetilde{\chi}^*$ both have same order 
 	         \cite[Theorem~$11.4.4$]{MR1625181}.
 	         Using the Davenport-Hasse theorem [cf. Section~\ref{preli}], we have 
 	         \begin{eqnarray} \label{R2}
 	         \frac{G_2 \big(\widetilde \chi \widetilde \chi_p' \big)}
 	         {G_2(\widetilde \chi)} 
 	         =
 	         \frac{G_1 \big(\widetilde \chi^* (\widetilde \chi_p')^* \big)^2}
 	         {G_1(\widetilde \chi^*)^2}.
 	         \end{eqnarray}
 	         Suppose that $\widehat{\mathbb{F}_p^\times}= \langle \chi_1 \rangle$, 
 	         the group of multiplicative characters of $\mathbb{F}_p$.
 	         Note that $\widetilde \chi^*$ has odd order $m$ and 
 	         $\widetilde \chi^* (\widetilde \chi_p')^*$ has order $2m$.
 	         We have $\widetilde \chi^*=\chi_1^b$ and 
 	         $\widetilde \chi^* (\widetilde \chi_p')^*=\chi_1^{\frac{b}{2}}$ 
 	         for some $b$. 
 	         Using the formula~(\ref{GKcoro}), we get that 
 	         $G_1(\chi_1^b)=(-p)^{\frac{b}{p-1}} \Ga_p \big(\frac{b}{p-1}\big)
 	         =(-p)^{\frac{1}{m}} \Ga_p \big(\frac{1}{m} \big)$
 	         and $G_1(\chi_1^{\frac{b}{2}})
 	         =(-p)^{\frac{1}{2m}} \Ga_p \big(\frac{1}{2m} \big)$.
 	         Hence, the desired result is obtained by equation~(\ref{R2}).
 	\end{enumerate}
 \end{proof}

 \begin{cor}
 	The quantity in (\ref{gap}) that determines $\varepsilon_p$ in the above theorem 
 	can be simplified as
 	\[
 	\Big\{ \frac{(x_0-1)!}{\big[\frac{x_0-1}{p} \big]!} 
 	\times 
 	\frac{\big[\frac{2x_0-1}{p} \big]!}{(2x_0-1)!}
 	\Big\}^2 \pmod{p},
 	\]
 	where $x_0$ is a solution of $2mx \equiv 1 \pmod{p}$.
 \end{cor}

 \begin{proof}
 Consider the following two congruence equations:
 \begin{eqnarray} 
    2mx &\equiv& 1 \pmod{p} \label{cong1} \quad \text{and} \\
    my &\equiv& 1 \pmod{p} \label{cong2}.
 \end{eqnarray}
 Both the congruence equations have an integer solution. 
 Note that $2x_0$ is a solution of (\ref{cong2}). By the property
 of $p$-adic gamma function, we have
 \[
   \Ga_p \big(\frac{1}{2m} \big) \equiv \Ga_p \big(x_0 \big) \pmod{p}
   \quad \text{and} \quad 
   \Ga_p \big(\frac{1}{m} \big) \equiv \Ga_p \big(2x_0 \big) \pmod{p}.
 \]
 Using the values of $p$-adic gamma function at integer points
 \cite[Chapter~$7$, \S $1.2$]{MR1760253}, we obtain:
 \begin{eqnarray*}
    \Big\{\Ga_p \big(\frac{1}{2m} \big) / 
    \Ga_p \big(\frac{1}{m} \big) \Big\}^2
    &\equiv& 
    \Big\{\Ga_p \big(x_0 \big) / 
    \Ga_p \big(2x_0 \big) \Big\}^2 \pmod{p}  \\
    &\equiv&
    \Big\{ \frac{(-1)^{x_0}(x_0-1)!}{\big[\frac{x_0-1}{p} \big]! p^{[(x_0-1)/p]}} 
    \times 
    \frac{\big[\frac{2x_0-1}{p} \big]! p^{[(2x_0-1)/p]}} {(-1)^{2x_0}(2x_0-1)!}
    \Big\}^2  \pmod{p}  \\
    &\equiv&
    \Big\{ \frac{(x_0-1)!}{\big[\frac{x_0-1}{p} \big]!} 
    \times 
    \frac{\big[\frac{2x_0-1}{p} \big]!}{(2x_0-1)!}
    \Big\}^2 \pmod{p}.
 \end{eqnarray*}
 In a special case where $2x_0<p+1$, the above quantity is same as
 $\Big\{ \frac{1}{(2x_0-1) \cdots (x_0-2)} \Big\}^2 \pmod{p}$.
 \end{proof}

 If the order of $\widetilde{\chi}$ is even and that divides $p+1$, then
 we also get the same result as part $(1)$ of Theorem~\ref{evenodd} 
 using the Stickelberger's theorem, which shows the consistency of our result. 
 We now prove Theorem~\ref{appstkl}.

 \begin{proof}
 	We split the proof into two cases. First assume that $m$ is odd.
 	We now consider primes $p \equiv 1 \pmod{4}$. 
 	Using Stickelberger's theorem \cite[Theorem~$5.16$]{MR1294139},
 	we have $G(\widetilde \chi,\widetilde \phi)=p$. 
 	Since $\widetilde{\chi}_p'$ has order $2$,
 	the order of $\widetilde{\chi} \widetilde{\chi}_p'$ is $2m$.
 	Write $p=4k+1$ for some $k \in \N$. 
 	Thus, we obtain that $(p+1)/2m=(2k+1)/m$
 	is odd. Hence, by the same theorem we have 
 	$G(\widetilde{\chi}\widetilde{\chi}_p',\widetilde \phi)=-p$, as required.
 	When $p \equiv 3 \pmod{4}$ with $m$ odd, we cannot apply the Stickelberger's theorem
 	to find $G(\widetilde{\chi}\widetilde{\chi}_p',\widetilde \phi)$ as
 	the quantity $(p+1)/2m$ is not odd.
 	
 	Next we deal with the case where $m$ is even. Thus, the order of $\widetilde{\chi}
 	\widetilde{\chi}_p'$ is $m$. For primes $p \equiv 1 \pmod{4}$, 
 	we have $(p+1)/m=2(2k+1)/m$ is odd. 
 	Hence, by Stickelberger's theorem we obtain that
 	$G(\widetilde{\chi},\widetilde \phi)=
 	G(\widetilde{\chi} \widetilde{\chi}_p',\widetilde \phi)=-p$, as desired. 
 	In a similar way we can show that $\varepsilon_p=1$, 
 	when $p \equiv 3 \pmod{4}$ with $(p+1)/m$ odd.
 \end{proof}

\subsubsection{The case $K|\Q_p$ ramified}
 As $p$ is odd, the possibilities for $K$ are $\Q_p(\sqrt{-p})$
 and $\Q_p(\sqrt{-p\zeta_{p-1}})$ depending on $(p, K|\Q_p)= 1$ or $-1$
 respectively. We can choose $\pi=\sqrt{-p}$ or 
 $\sqrt{-p \zeta_{p-1}}$ as a uniformizer of $K$ and write
 $K=\Q_p(\pi)$.

 Since $K|\Q_p$ is ramified quadratic, we have 
 $N_{K|\Q_p}(\sO_K^\times)=\Z_p^{\times2}$.
 In this case, we have $\chi_p'|_{\sO_K^\times}=1$ (i. e.,  $\chi_p'$ is unramified). 
 Choose an additive character $\phi$ of 
 conductor $0$ and a normalized Haar measure $dx$.
 Since $K|\Q_p$ is ramified, the conductor of 
 $\phi_K=\phi \circ \mathrm{Tr}_{K|\Q_p}$ is $1$ [cf. Lemma~\ref{tunnel}].
 We now prove Theorem~\ref{main} in the ramified case.
 
 \begin{proof}
 Since $\chi_p'$ is unramified, 
 using the property $(2)$ of local $\varepsilon$-factors, 
 we get that
 \begin{eqnarray} \label{deligne}
   \varepsilon(\chi \chi_p', \phi_K,dx)=
   \chi_p(N_{K|\Q_p}(\pi))^{a(\chi)+1} \cdot \varepsilon(\chi, \phi_K,dx).
 \end{eqnarray}
 If the conductor of $\chi$ is odd we get that the number $\varepsilon_p=1$.
 So assume that $a(\chi)$ is even. The epsilon factor is thus given by
 the quantity:
 \begin{eqnarray} \label{variation}
   \varepsilon_p=\chi_p(N_{K|\Q_p}(\pi))=\Big( \frac{N(\pi)/p}{p} \Big).
 \end{eqnarray}
 Since $N_{K|\Q_p}(\pi)=-\pi^2$, we obtain that
 $
   \varepsilon_p=\Big( \frac{-\pi^2/p}{p} \Big).
 $
 Therefore when $p$ is odd, we deduce that:
 \begin{equation} \label{residuesymbol}
    \varepsilon_p =
    \begin{cases}
       \Big(\frac{1}{p}\Big)=1, 
       \quad \text{if} \,\, (p,K|\Q_p)=1,  \\
       \Big(\frac{\zeta_{p-1}}{p}\Big)=\Big( \frac{-1}{p} \Big),
       \quad \text{if} \,\, (p,K|\Q_p)=-1.
    \end{cases}
 \end{equation}
 \end{proof}

 Let $\varepsilon(f)$ be the global $\varepsilon$-factor associated to $f$
 and $\varepsilon_p(f)$ be its $p$-part. 
 For the character $\chi_p$ defined before, 
 the newform twisted  by $\chi_p$ is denoted by $f \otimes \chi_p$. 
 Then we have the following classification of the local data of a newform.

 \begin{cor}
  \label{mainkoro}
 	Let $\pi_p=\pi_{f,p}$ be the local component at $p$ of a $p$-minimal newform $f$. 
 	We have
 	\begin{enumerate}
 		\item 
 		$\pi_p$ is Steinberg if $N_p=1$ and $C_p=0$.
 		\item
 		$\pi_p$ is principal series if $N_p \geq 1$ with $N_p=C_p$. 
 		\item
 		If $\pi_p$ is not of the above type, then it is supercuspidal.
 		In this case, we have $N_p > C_p$. For odd $p$, it is always induced
 		by a quadratic extension $K|\Q_p$. 
 		If $N_p \geq 2$ is even, then $K$
 		is the unique unramified quadratic extension of $\Q_p$. In the case of
 		$N_p \geq 3$ odd with $p \equiv 3 \pmod{4}$, we have
 		\begin{enumerate}
 			\item 
 			$K=\Q_p(\sqrt{-p})$ if $\varepsilon(f \otimes \chi_p)=
 			\chi_p(N')\varepsilon(f)$.
 			\item 
 			$K=\Q_p(\sqrt{-p\zeta_{p-1}})$ if $\varepsilon(f \otimes \chi_p)=
 			-\chi_p(N')\varepsilon(f)$.
 		\end{enumerate}
 		The same cannot be concluded, when $p \equiv 1 \pmod{4}$.
 	\end{enumerate}
 \end{cor}
 
 \begin{proof}
 	Depending upon $N_p$ and $C_p$, the local component at $p$ of a newform
 	has been classified in \cite[prop. $2.8$]{MR2869056}.
 	We just need to prove the relation 
 	$\varepsilon(f \otimes \chi_p)=\chi_p(N') \varepsilon(f)
 	\varepsilon_p$ to complete the proof. 
 	Note that the number $\varepsilon_q (p \neq q)$ is determined by
 	$
 	  \varepsilon_q=\big(\frac{q}{p} \big)^{N_q}
 	$
 	\cite[Theorem $3.2$, part $(1)$]{MR3056552},
 	where $N_q$ denote the exact power of $q$ that divides $N$.
 	By hypothesis, we have 
 	$\varepsilon_p(f \otimes \chi_p)=\varepsilon_p(f) \varepsilon_p$.
 	Running over all primes $p$, we get that $\varepsilon(f \otimes \chi_p)=
 	\varepsilon(f) \prod_{p}^{} \varepsilon_p
 	=
 	\varepsilon(f) \varepsilon_p \prod_{q \neq p, q|N}^{} 
 	\Big( \frac{q^{N_q}}{p} \Big)
 	=
 	\chi_p(N') \varepsilon(f) \varepsilon_p$.
 \end{proof}

 \begin{remark}
 	The classification of the local data at $p$ of a newform determined by
 	$N_p$ and $C_p$ does not distinguish the quadratic extensions of $\Q_p$,
 	the local component $\pi_p$ at $p$ is induced from in the supercuspidal case.
 	The above corollary does that in terms of the variation 
 	of $\varepsilon$-factor of $f$.
 \end{remark}

\subsection{The case $p=2$}
 For $p=2$, more representations of the Weil group are involved
 and it can be non-dihedral. When inertia acts reducibly,
 the local representation $\rho_2(f)$ is dihedral; otherwise it has projective image 
 isomorphic to one of three ``exceptional" groups $A_4,S_4$ or $A_5$.
 The $A_5$ case cannot occur since the Galois group
 $\mathrm{Gal}(\overline{\Q}_2|\Q_2)$
 is solvable. Weil proved in \cite{MR0379445} that over $\Q$,
 the $A_4$ case also does not occur, 
 so $S_4$ is the only possibility of the projective image of $\rho_2(f)$.
 In this case, the corresponding field extension of $\Q_2$ is obtained by
 adjoining the coordinates of the $3$-torsion points of the following 
 elliptic curves \cite[Section $6$]{MR2223979}:
 \[
   E_+^{(r)}: rY^2=X^3+3X+2 \quad \text{and} \,\, E_-^{(r)}: rY^2=X^3-3X+2
   \quad \text{with} \,\,r \in \{\pm 1, \pm 2\}.
 \]
 In local cit., the $2$-adic valuation 
 of the level of the modular form is given as follows:
 $7$ for the curve $E_+^{r}$ with $r$ as above, $3$ for the curve
 $E_-^{(-1)}$, $4$ for the curve $E_-^{(1)}$ and $6$ when the curve is
 $E_-^{(\pm 2)}$. 
 Hence, if $p=2$ is a non-dihedral supercuspidal prime, then we have 
 $N_2 \in \{3,4,6,7\}$.
 
 From above we see that non-dihedral supercuspidal case can occur only in $8$ cases. 
 In all such cases local root number can be computed and it can be found in 
 \cite[Remark $11$]{MR3056552} and \cite[Remark $22$]{DPT}.
 
 From now on we assume that $p=2$ is a dihedral supercuspidal prime. In this case,
 the representation $\rho_2(f)$ is induced from a quadratic extension $K|\Q_2$.
 Note that there are seven quadratic extensions 
 $\Q_2(\sqrt{t})$ of $\Q_2$ with $t=-3,-1,3,2,-2,6,-6$. 
 Among them $\Q_2(\sqrt{-3})$ is unramified and rest of them are ramified.
 Among ramified extensions, two of them (corresponding to $t=-1,3$) 
 have discriminant with valuation $2$
 and rest of them have discriminant with valuation $3$.
 Let $d$ denote the valuation of the discriminant of $K|\Q_2$. 
 Thus, we have $d \in \{ 2,3 \}$.
 Here, $\pi=\sqrt{t}$ is the uniformizer of $K$ and we write $K=\Q_2(\pi)$.
 The following is the analogue of Proposition~\ref{np} for $p=2$.

 \begin{prop} \label{n2}
 	Let $p=2$ be a dihedral supercuspidal prime. 
 	\begin{enumerate}
 		\item 
 		     If $K$ is unramified, then $N_2$ is even.
 	    \item
 	         Assume that $K|\Q_2$ is ramified with $l(\chi) \geq d$. 
 	         We have $N_2$ is odd if $\chi$ is minimal; otherwise $N_2$ is even.  
 	\end{enumerate}
 \end{prop}
 
 \begin{proof}
 	The relation~(\ref{relation}) for $p=2$ gives us that
 	\begin{eqnarray} 
	\label{eqnevenodd}
 	   N_2 = 
 	   \begin{cases}
 	       2 a(\chi), & \quad \text{if} \,\, K|\Q_2 \,\,
 	       \text{is unramified}, \\
 	       2+a(\chi), & \quad \text{if} \,\, K|\Q_2 \,\,
 	       \text{is ramified with valuation} \,\, 2, \\
 	       3+a(\chi), & \quad \text{if} \,\, K|\Q_2 \,\,
 	       \text{is ramified with valuation} \,\, 3.
 	   \end{cases}
 	\end{eqnarray}
 	Thus, $N_2$ is even when $K$ is unramified.
 	
    In the ramified case, we first assume that $\chi$ is minimal.
    Using \cite[\S $41.4$ Lemma]{MR2234120}, 
    we have $l(\chi) \geq d-1$. Moreover, if  $l(\chi) \geq d$, 
    then applying same lemma we get $l(\chi) \not\equiv d-1 \pmod{2}$.
    If $d=2$, then we conclude that $l(\chi)$ is even and so $a(\chi)=l(\chi)+1$
    is odd.
    When $d=3$, we conclude that $l(\chi)$ is odd and so $a(\chi)=l(\chi)+1$
    is even.

    We now prove that if $\chi$ is not minimal, then $N_2$ is even. To prove that, 
   consider a non-minimal character $\chi$.  As usual, we write 
    $\chi=\theta_K\chi'$ \cite[Section $41.4$]{MR2234120} 
    with $\theta$ a character of $\Q_2^\times$ and
    $\chi'$ minimal over $\Q_2$. 
    As in the proof of Proposition~\ref{np}, we deduce that $a(\chi') < a(\chi)$ and $a(\theta) \geq 1$.
  Observe that $a(\chi)=a(\theta_K)$
    This follows from  $a(\chi)=a(\chi' \theta_K) \leq \text{max}(a(\chi'),a(\theta_K))$
    with equality if $a(\chi') \neq a(\theta_K)$.

    Since $K$ is wildly ramified, $\omega_{K|\Q_2}$ has conductor $d$
    \cite[Section $41.3$]{MR2234120},  we conclude that  $a(\omega_{K|\Q_2})=2$ or $3$.
    Consider first the case  $d=2$. 
    If $a(\theta) \leq d=2$, then by Lemma~\ref{tunnel}, we have
    $a(\theta_K) \leq a(\theta) \leq 2$.
    Since $\chi'$ is minimal over $\Q_2$, we obtain that $l(\chi') \geq d-1$
    \cite[\S $41.4$ Lemma]{MR2234120}.

    Let us first assume  $l(\chi')=d-1=1$. We  then have 
    $a(\chi')=2$. Both $a(\chi')=2$ and $a(\theta) \leq 2$ (i.e, $a(\theta_K) \leq 2$)
    cannot occur simultaneously as it will contradict the fact 
    $a(\chi') < a(\chi)=a(\chi'\theta_K)$.
    Hence, we conclude that $a(\theta)>2$. 
    
    We now assume that $l(\chi') \geq d=2$. As above, the minimality of $\chi'$ gives us
    that $l(\chi') \geq 2$ is even and so $a(\chi') \geq 3$ is odd. 
    In this case, if $a(\theta) \leq 2$, that is, $a(\theta_K) \leq 2$,
    then we obtain that $a(\chi)=a(\chi'\theta_K)=a(\chi')$ as $a(\chi') \geq 3$, 
    contradicts the non-minimality of $\chi$. Hence, in this case also $a(\theta)>2$.

    In both cases we obtain $a(\theta) >2= a(\omega_{K|\Q_2})$ so that 
    $a(\theta \omega_{K|\Q_2}) = a(\theta)$. By Lemma~\ref{tunnel}, we get that
    $a(\theta_K)=a(\theta)+a(\theta \omega_{K|\Q_p})-a(\omega_{K|\Q_p})=
    2(a(\theta)-1)$ and so $a(\chi)=a(\theta_K)$ is even.  From equation~\ref{eqnevenodd}, we 
    conclude that if $\chi$ is non-minimal then $N_2$ is even. 
    
    When $d=3$, we can prove similarly that
    $a(\theta_K)=2 a(\theta)-3$ and hence  $a(\chi)=a(\theta_K)$ is odd.
  \end{proof}

 We now prove Theorem~\ref{main} and Theorem~\ref{appstkl} for $p=2$.
 \begin{proof}
 	The part $(1)$ of Theorem~\ref{main} is valid
  	for the unramified supercuspidal prime $p=2$ also. 
  	Thus, we have $\varepsilon_2=1$, when $a(\chi)>1$.
 	When $a(\chi)=1$, note that $\widetilde \chi$ is a character of
 	$\mathbb{F}_{2^2}^\times$ of order $3$. 
 	Since $\widetilde \chi_2$ has order $2$, the order of $\widetilde \chi
 	\widetilde \chi_2$ is also $3$. 
 	Using Stickelberger's theorem \cite[Theorem~$5.16$]{MR1294139}, we have
 	$G(\widetilde \chi,\widetilde \phi)=
 	G(\widetilde \chi \widetilde \chi_2,\widetilde \phi)=2$
 	and hence $\varepsilon_2=1$, completing the proof of Theorem~\ref{appstkl} for $p=2$.

 	We now prove Theorem~\ref{main} in the ramified case. 
    By proposition above we see that if $\chi$ is minimal,
    then $a(\chi)$ is odd, when $d=2$, and it is even, when $d=3$.
    If $\chi$ is not minimal, 
    then $a(\chi)$ is even, when $d=2$, and it is odd, when $d=3$.
    As odd primes [cf. Equations~\ref{deligne} and \ref{variation}], 
    we have $\varepsilon_2=1$, if $a(\chi)$ is odd;
    otherwise we get $\varepsilon_2=\chi_2(N_{K|\Q_2}(\pi))$.
    Hence, we get the number $\varepsilon_2$ as in (\ref{eps2}).
 \end{proof}
 
 As before [cf. Corollary~\ref{mainkoro}], the type of the local component $\pi_{f,2}$
 can be classified by $N_2$ and $C_2$ ( with $N_2 \in \{3,4,6,7\}$ in the 
 non-dihedral supercuspidal case).
 We now classify the quadratic extensions $K|\Q_2$,
 the local component $\pi_{f,2}$ is induced from in the dihedral supercuspidal case.
 Note that $\Q$ has three quadratic extensions ramified only at $2$
 (having absolute discriminant a power of $2$), namely $\Q(i),\Q(\sqrt{2})$
 and $\Q(\sqrt{-2})$. Their corresponding quadratic extensions will be
 denoted by $\chi_{-1}, \chi_2$ and $\chi_{-2}$ respectively.
 Let $\chi$ be the quadratic character associated by class field theory
 to any of these characters $\chi_i$ for $i=-1,2,-2$.
 Using \cite[Theorem $4.2$, part $(1)$]{MR3056552}, we have the following two relations
 as odd primes [cf. Corollary~\ref{mainkoro}]:
 \begin{enumerate}[I.]
 	\item 
 	     $\varepsilon(f \otimes \chi)=\chi(N') \varepsilon(f)$, if $\varepsilon_2=1$,
 	\item
 	     $\varepsilon(f \otimes \chi)=-\chi(N') \varepsilon(f)$, if $\varepsilon_2=-1$.
 \end{enumerate}
 
 \begin{cor} \label{corop=2}
 	Let $p=2$ be a dihedral supercuspidal prime for $f$. 
 	Then $\pi_{f,2}$ is always induced by a quadratic extension $K|\Q_2$.
 	If $f$ is $2$-minimal and $N_2 \geq 2$ is even, 
 	then $K$ is the unique unramified quadratic extension 
 	of $\Q_2$. In the ramified case, we have the following classifications of $K$:
    \begin{center}
 	   \begin{tabular}{ |l|l|l| }
 		   \hline
 		   \multicolumn{3}{ |c| }{\textbf{Classification of $K$ for $p=2$}} \\
 		   \hline
 		   $p$-minimality of $f$ & $K=\Q_2(\sqrt{t})$ & Property \\ \hline
 		   \multirow{2}{*}{Yes} & $t=-1,-2,2,3$ & I \\
 		   & $t=-6,6$ & II \\ \hline
 		   \multirow{2}{*}{No} & $t=-1,-2,2,-6,6$ & I \\
 		   & $t=3$ & II \\ 
 		   \hline
 	   \end{tabular}
    \end{center}
 \end{cor}
 
 \begin{remark}
 	If $f$ is not $2$-minimal, then we cannot distinguish whether the extension 
 	$K|\Q_2$ is unramified or ramified (since $N_2$ is even in both cases).
 	When $f$ is a $2$-minimal newform, the extension $K|\Q_2$ can be distinguished 
 	by the parity of $N_2$ ($K$ is unramified if $N_2$ is even and $K$ is ramified
 	if $N_2$ is odd). 
 \end{remark}

 \bibliographystyle{crelle}
 \bibliography{Eisensteinquestion.bib}
\end{document}